\documentclass[10pt]{amsart}
\usepackage{mathrsfs} % for caligraphic L for Lagrangian
\usepackage[all]{xy} % for commutative diagram

\usepackage{graphicx} % for pdfTeX

\usepackage{tikz}
\usetikzlibrary{arrows}
\usetikzlibrary{shapes,snakes}

\newtheorem{thm}{Theorem}[section]
\newtheorem{prop}[thm]{Proposition}
\newtheorem{conj}[thm]{Conjecture}

\newtheorem{lem}[thm]{Lemma}

\theoremstyle{definition}

\numberwithin{equation}{section}
\newtheorem{rem}[thm]{Remark}
\newtheorem{rems}[thm]{Remarks}
\newtheorem{ex}[thm]{Example}

\def\one{\mathbf 1}

\newcommand{\sgn}{\mbox{sign}}

\begin{document}
\bibliographystyle{amsalpha}

\title[Asymptotic Sign Coherence]
{
Asymptotic Sign Coherence Conjecture
}

\author{Michael Gekhtman}
\address{\noindent 
Department of Mathematics, University of Notre Dame, Notre Dame, IN 46556,
USA}
\email{mgekhtma@nd.edu}
\author{Tomoki Nakanishi}
\address{\noindent Graduate School of Mathematics, Nagoya University, 
Chikusa-ku, Nagoya,
464-8604, Japan}
\email{nakanisi@math.nagoya-u.ac.jp}

%\subjclass[2010]{Primary 13F60, Secondary 33E20}
%\keywords{dilogarithm, quantum dilogarithm, cluster algebra}
%\thanks{This work was partially supported by JSPS KAKENHI Grant Number 16H03922.}

\date{}
\begin{abstract}
The sign coherence phenomenon is an important feature of $c$-vectors in cluster algebras with principal coefficients.
In this note, we consider a more general version of $c$-vectors defined for arbitrary cluster algebras of geometric type and formulate a conjecture describing their asymptotic behavior. This conjecture, which is called {\em the asymptotic sign coherence conjecture}, states that for any infinite sequence of matrix mutations that satisfies certain natural conditions, the corresponding $c$-vectors eventually become sign coherent. We prove this conjecture for rank $2$ cluster algebras of infinite type and for a particular sequence of mutations in a cluster algebra associated with the Markov quiver. 

%{\bf (ver.190327: not for distribution)}
\end{abstract}

\maketitle

\section{Introduction}

%\subsection{Background and motivation}
%\label{subsec:back1}

A study of $c$-vectors for cluster algebras with principal coefficients was initiated in \cite{CAIV}, the fourth in the series of foundational papers that gave rise to the theory of cluster algebras. There, $c$-vectors appeared, together with related concepts such as $g$-vectors and $F$-polynomials, as tools for understanding a deeper structure of cluster variables. Also in \cite{CAIV}, a conjecture that later became known as {\em the sign coherence conjecture} was first formulated. It states that in any cluster mutation equivalent to the initial one, each $c$-vector is nonzero and has either all non-negative or all non-positive coefficients. It was soon realized that this conjecture has important implications in various aspects of the theory of cluster algebras, notably, in establishing duality properties as was done in \cite{NZ}. The conjecture was  proved in \cite{DWZ} for the skew-symmetric cluster algebras 
and in \cite{GHKK} for the skew-symmetrizable ones. 

It is natural to wonder if there is an analogue of sign coherence that remains valid if the condition on coefficients being principal is relaxed, in particular, if a similar phenomenon occurs in  cluster algebras of {\em geometric type}  in the sense of \cite[Def.2.12]{CAIV},  where their coefficients take value in tropical semifields.
Clearly, this property cannot be satisfied as is --- one can simply start with the initial coefficients for which it is not valid. This note proposes how sign coherence can be treated in arbitrary cluster algebras of geometric type: after  a sufficiently generic sequence of mutation, 
$c$-vectors, defined in a more general context, become sign coherent.

In the next section, after providing a background on matrix mutations and providing an example illustrating the phenomenon described above, we formulate a conjecture that describes this behavior. We call it
{\em the asymptotic sign coherence conjecture}. 
In section 3, we verify our conjecture in the rank $2$ case. The final section deals with a particular sequences of mutations in a cluster algebra associated with the Markov quiver - we show that the conjecture holds true in this case as well.

 %\subsection{Outline and main results}
% Let us briefly describe the outline and the main results of the paper.

 %\bigskip
{\em Acknowledgements.} 
This work is supported in part by JSPS Grant No. 16H03922 to T. N. and by M.G.'s NSF grant No. 1702054.

\section{Preliminaries and the main conjecture}

\subsection{Matrix mutations.
%Mutations in cluster algebras
}
\label{subsec:mut1}

%Let us recall two main notions in cluster algebras,
%namely, a {\em seed} and its {\em mutation}.
%See  \cite{Fomin02,Fomin07} for more information
%  on cluster algebras.

%Let us fix a positive integer $n$ throughout the paper.
We say that an $N\times N$ integer matrix $B=(b_{ij})_{i,j=1}^N$
is
{\em skew-symmetrizable}
if there is  a diagonal matrix $D=\mathrm{diag}(d_1,\dots,d_N)$
with positive integer diagonal
entries $d_1,\dots,d_N$ such that $DB$ is skew-symmetric,
i.e., $d_ib_{ij}= - d_j b_{ji}$.
We call such $D$ a {\em skew-symmetrizer} of $B$. $B$ is called {\em irreducible} if there does not exist  a pair $I,J$ of nonempty subsets of $[N]:=\{1,\ldots, N\} $ such that
$b_{ij}=0$ for any $i\in I$ and $j\in J$. If $B$ is skew-symmetric, it can be realized as an adjacency matrix of a quiver with $N$ vertices. The irreducibility property in this case is equivalent to connectedness of the quiver.

The notion of a {\em matrix mutation} is one of the key ingredients in the definition of a cluster algebra. It is also the only ingredient of that definition that we will need in this note. We will be interested in  matrix mutations of integer matrices with skew-symmetrizable principal parts. Namely, let $\hat B$ be an integer $(N+M)\times N$ matrix such that its principal submatrix $B:=\hat B_{[N]}$ formed by the first $N$ rows is
skew-symmetrizable. Columns of $\hat B$ determine the rules of transformations of cluster variables $x_1,\ldots, x_N$ with the variables $x_{N+1},\ldots, x_{N+M}$ being frozen. Alternatively, we may view the columns of the bottom  $M\times N$ submatrix of $\hat B$ as encoding expressions for coefficients in a cluster algebra of geometric type. See \cite[Section 2]{CAIV} for the explanation.

%Let us fix a semifield $\mathbb{P}$,
%that is,  an abelian multiplicative group
%with a binary operation $\oplus$ called the {\em addition},
%which is commutative, associative, and distributive, i.e., $a(b\oplus c)=
%ab \oplus ac$.
%Let $\mathbb{Z}\mathbb{P}$ be the group ring of $\mathbb{P}$.
%Since $\mathbb{Z}\mathbb{P}$ is a domain \cite{Fomin02},
 %the field of fractions of $\mathbb{Z}\mathbb{P}$
% is well defined
%and we denoted it by $\mathbb{Q}\mathbb{P}$.
%Let $\mathcal{F}=\mathcal{F}_{\mathbb{P}}$ be a purely transcendental
%field extension of $\mathbb{Q}\mathbb{P}$ of degree $n$,
%that is $\mathcal{F}$ is isomorphic to
%a rational function field of $n$ variables
%with coefficients in $\mathbb{Q}\mathbb{P}$.
%We call the semifield $\mathbb{P}$ and
%the field $\mathcal{F}$ the {\em coefficient semifield}
%and the {\em ambient field} (of a cluster algebra
%under consideration), respectively.

%A {\em seed with coefficients in $\mathbb{P}$}
%is
 %a triplet $(B,x,y)$ consisting of
%an $n\times n$ skew-symmetrizable matrix $B$,
%an $n$-tuple  $(x_i)_{i=1}^n$ of algebraically independent 
%elements in $\mathcal{F}$,
%and an  $n$-tuple  $(y_i)_{i=1}^n$ of elements in $\mathbb{P}$.
%For each $k=1,\dots,n$, the {\em mutation of a seed $(B,x,y)$ at $k$}
%s another seed $(B',x',y')=\mu_k(B,x,y)$,
%which is obtained from $(B,x,y)$ by the following formulas:

For each $k=1,\dots,N$, the {\em mutation of  $\hat B$ at $k$}
is another integer $(N+M)\times N$ matrix $\hat B'=\mu_k(\hat B)$,
which is obtained from $\hat B$ by the following formulas:

\begin{align}
\label{eq:Bmut1}
b'_{ij}
&=
\begin{cases}
-b_{ij}
& \text{$i=k$ or $j=k$}\\
b_{ij}+[-\varepsilon b_{ik}]_+ b_{kj}
+ b_{ik} [\varepsilon b_{kj}]_+
&
i,j\neq k,
\end{cases}
%\\
%\label{eq:xmut1}
%x_i '
%&= 
%\begin{cases}
%\displaystyle
%x_k^{-1} \bigg(\prod_{j=1}^n x_j^{[-
%\varepsilon b_{jk}]_+}\bigg)
%\frac{1+\hat{y}_k^{\varepsilon}}
%{ 1\oplus y_k^{\varepsilon}} &i=k\\
%x_i
%& i\neq k,\\
%\end{cases}
%\\
%\label{eq:ymut1}
%y_i '&= 
%\begin{cases}
%y_k^{-1} &i=k\\
%y_i y_k^{[\varepsilon b_{ki}]_+}(1\oplus y_k^{\varepsilon})^{-b_{ki}}
%& i\neq k,\\
%\end{cases}
%\end{align}
%where
%\begin{align}
%\label{eq:yhat2}
%\hat{y}_i := y_i \prod_{j=1}^n x_j^{b_{ji}},
\end{align}
where, for $a\in\mathbb{R}$, we denote $[a]_+=\max(a,0)$, and $\varepsilon$ is a sign, $+$ or $-$, which is naturally identified with $1$ or $-1$,
respectively.
Then we have the following properties:
\par
(1). The right hand side of \eqref{eq:Bmut1}
is  independent of
the choice of sign $\varepsilon$.
\par
(2).
If $D$ is a skew-symmetrizer of $B$,
then it is also a skew-symmetrizer of  $B'=\hat B'_{[N]}$.
\par
(3).
The mutation $\mu_k$ is involutive, namely,
\begin{align}
\nonumber
%\label{eq:inv1}
\mu_k \circ \mu_k=\mathrm{id}.
\end{align}
More generally, if $\hat B'$ is obtained from $\hat B$ by a sequence of mutations \eqref{eq:Bmut1}, then $\hat B'$ is said to be {\em mutation equivalent} to $\hat B$.

If $M=N$ and the block formed by the last $N$ rows of $\hat B$ is equal to the identity matrix $\one_N$ then the cluster algebra associated with $\hat B$ is said to have {\em principal coefficients}. In this case, the bottom $N\times N$ submatrix of any matrix $\hat B'$ mutation equivalent to $\hat B$ is called a $C$-matrix and its columns are called $c$-vectors.

\subsection{Main conjecture}
\label{subsec:conj}

The following key property of $c$-vectors that proved to be of fundamental importance in the theory and applications of cluster algebras was conjectured in \cite[Conjecture 5.5 $\&$ Proposition 5.6]{CAIV} and later proved in  \cite[Theorem 1.7]{DWZ}  in the skew-symmetric case  and \cite[Corollary
5.5]{GHKK} in the skew-symmetrizable case with both proofs using \cite[Proposition 5.6]{CAIV}.

\begin{thm}[{\bf Sign-coherence of c-vectors}]
\label{sign_coh}
Each $c$-vector  is a nonzero vector,
and its components are either all non-negative or all non-positive.
\end{thm}

In this note, we are interested in a behavior of a more general version of $c$-vectors: namely, we will denote by $C$ the bottom $M\times N$ submatrix of $\hat B$  and refer to its columns as $c$-vectors. Whenever we will need to invoke the original definition of $c$-vectors, we will call them {\em principal} $c$-vectors.

We begin with the following example.
\begin{ex}
\label{random} 
Let $B$ be the adjacency matrix for the  quiver below. This quiver is a rather randomly chosen one with four vertices.
%\begin{figure}[h]
%\begin{center}
%\includegraphics[width=4cm, height=4cm]{AsymptSignCohEx.pdf}
%\end{center}
%\caption{}
%\label{example_quiver}
%\end{figure}
\medskip

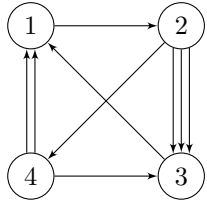
\begin{center}
\begin{tikzpicture}
\tikzset{vertex/.style = {shape=circle,draw, minimum size=1.5em}}
\tikzset{edge/.style = {->,> = latex'}}
% vertices
\node[vertex] (1) at  (0,0) {$1$};
\node[vertex] (2) at  (2,0) {$2$};
\node[vertex] (3) at  (0,-2) {$4$};
\node[vertex] (4) at  (2,-2) {$3$};
%edges
\draw[edge] (1) to (2);
\draw[edge] (3.100) to (1.260) ;
%\draw[edge] (1.325) to (4.125);
\draw[edge] (2) to (3);
\draw[edge] (2) to (4);
\draw[edge] (2.290) to (4.70);
\draw[edge] (2.250) to (4.110);

\draw[edge] (3.80) to (1.280) ;
\draw[edge] (3) to (4);
\draw[edge] (4) to (1);
\end{tikzpicture}
\end{center}

\medskip

Consider a sequence, also rather randomly chosen,  of $C$-matrix mutations, together with the initial $C$-matrix $C[0]$, specified below:
%$C$-matrix 
%\[
%C[0]=\begin{pmatrix} 1 & -1 & 2 &1\\ -2 & 1 & 3 & -2\\ -1 & 1 & -1 & 0\\ 0 & -1 & -1 & - 2 \end{pmatrix}
%\]

\begin{small}
\begin{equation*}
\begin{split}
&C[0]=\begin{pmatrix} 1 & \text{-}1&2&1\\ \text{-}2&1&3 &\text{-}2\\ \text{-}1&1&\text{-}1&0\\0&\text{-}1&\text{-}1&\text{-} 2 \end{pmatrix}\ {\stackrel{1} {\to} }\ 
C[1]=\begin{pmatrix} \text{-}1 & 0 & 2 &1\\ 2 & 1 & 1 & \text{-}6\\ \text{-}1 & 2 & \text{-}1 & 0\\ 0 & \text{-}1 & \text{-}1 & \text{-}2 \end{pmatrix} \ {\stackrel{3} {\to} }\\
&C[2]=\begin{pmatrix} \text{-}1 & 0 & \text{-}2 &1\\ 2 & 1 & \text{-}1 & -6\\ \text{-}2 & 0 & 1 & \text{-}1\\ \text{-}1& \text{-}3 & 1 & \text{-}3 \end{pmatrix}\  {\stackrel{4} {\to} } \ 
C[3]=\begin{pmatrix}\text{-}1&3&\text{-}2&\text{-}1\\ \text{-}10&1&\text{-}7&6\\ \text{-}4&0&0&1\\ \text{-}7&\text{-}3&\text{-}2&3\end{pmatrix} {\stackrel{2} {\to} }\\
&C[4]=\begin{pmatrix} \text{-}1 & \text{-}3 & \text{-}2 &8\\ \text{-}10 & \text{-}1 & \text{-}7 & 9\\ \text{-}4 & 0 & 0 & 1\\ \text{-}22 & 3 & \text{-}17 & 3 \end{pmatrix} \ {{\stackrel{3} {\to} }}\ 
C[5]=\begin{pmatrix} \text{-}1 & \text{-}13 & 2 & 8\\ \text{-}10 & \text{-}36 & 7 & 9\\ \text{-}4 & 0 & 0 & 1\\ \text{-}22& \text{-}82 & 17 & 3 \end{pmatrix}\ {\stackrel{1} {\to} }\\
&C[6]=\begin{pmatrix} 1 & \text{-}23 & 2 & 8\\ 10 & \text{-}136 & 7 & 9\\ 4 & \text{-}40 & 0 & 1\\ 22 & \text{-}302 & 17 & 3 \end{pmatrix} \ {\stackrel{4} {\to} }\ 
C[7]=\begin{pmatrix} 105 & \text{-}23 & 114 & \text{-}8\\ 127 & \text{-}136 & 133 & \text{-}9\\ 17 & \text{-}40 &14 & \text{-}1\\ 61 & \text{-}302 & 59 & \text{-}3 \end{pmatrix}\  {\stackrel{1} {\to} }\\
&C[8]=\begin{pmatrix} \text{-}105 & \text{-}23 & 114 & 1357\\ \text{-}127 & \text{-}136 & 133 & 1642\\ \text{-}17 & \text{-}40 & 14 & 220\\ \text{-}61& \text{-}302 & 59 & 790 \end{pmatrix}\ {\stackrel{\cdots} {\to} }
\end{split}
\end{equation*}
\end{small}

\end{ex}

Example above is just one of many we have considered and all of these examples exhibited the same phenomenon: {\em after sufficiently many random mutations, $C$-vectors become sign-coherent}.

To frame these observations as an explicit conjecture, we will only need to consider cluster algebras with a single frozen variable since mutations of rows of a $C$-matrix are independent of  each other by \eqref{eq:Bmut1}.
We fix a skew-symmetrizable $N\times N$ matrix $B$ and an integer vector $\mathbf a = (a_1,\ldots, a_N)$ and define an $(N+1)\times N$ 
%extended exchange 
matrix $\hat B$ by appending $\mathbf a$ to $B$ as the $(N+1)$th row. 
%In addition, let $\hat B_{pr}$ denote an $N\times 2N$ matrix $\begin{bmatrix} B\\ \ \one_N\end{bmatrix}$ (in other words, $\hat B_{pr}$ is the initial exchange matrix for the cluster algebra with principal coefficients defined by $B$).
%If ${\hat B}'$ is mutation-equivalent to $\hat B$, we denote by $\hat B'_{pr}$ the matrix obtained by applying the same sequence of mutations to $\hat B_{pr}$
%and define the distance between $\hat B_{pr}$ and ${\hat B_{pr}}'$, denoted by $\mbox{dist}(\hat B_{pr}, {\hat B_{pr}}')$ as the minimal number of matrix mutations needed to obtain  ${\hat B_{pr}}'$ from $\hat B_{pr}$.

Let $\mu=\left ( \mu_{k_j} \right )_{j=1}^\infty$ ($k_j\in [1,N]$) be a sequence of matrix mutations applied to $\hat B$.
We denote $\mu_{k_n}\circ\cdots\circ\mu_{k_2}\circ\mu_{k_1}$ by $\mu^{(n)}$,  $\mu^{(n)} (\hat B)$ by ${\hat B}^{(n)}$ and the last row
of ${\hat B}^{(n)}$ by ${\mathbf a}^{(n)}= (a_1^{(n)},\ldots, a_N^{(n)})$. We use a convention ${\hat B}^{(0)}={\hat B}, {\mathbf a}^{(0)}={\mathbf a}$. Note that if ${\mathbf a}=0$, then ${\mathbf a}^{(n)}=0$ for any $n$ by \eqref{eq:Bmut1}. Therefore, from now on we assume that ${\mathbf a}$ is a {\em nonzero} vector.  

Let $\hat B_{pr}$ denote an $N\times 2N$ matrix $\begin{bmatrix} B\\ \ \one_N\end{bmatrix}$ (in other words, $\hat B_{pr}$ is the initial exchange matrix for the cluster algebra with principal coefficients defined by $B$). We define {\em the distance} $\mbox{dist}(\hat B_{pr}, \mu^{(n)}({\hat B_{pr}})) \leq n$  as the {\em minimal} number of matrix mutations needed 
to obtain  $\mu^{(n)}({\hat B_{pr}})$ from $\hat B_{pr}$. This notion of  distance agrees with {\em the distance in the exchange graph of labelled seeds} in the cluster algebra associated with $B$ (with any coefficients), see, e.g., \cite{synchr}.

We say that the sequence of mutatons $\mu$ is {\em monotone} (with respect to $B$) if 
\begin{equation}
\label{monot_seq}
\mbox{dist}(\hat B_{pr}, \mu^{(n+1)}({\hat B_{pr}})) > \mbox{dist}(\hat B_{pr}, \mu^{(n)}({\hat B_{pr}})) \ (n=1,2,\ldots).
\end{equation}

%Here, $\hat B_{pr}$ denote an $N\times 2N$ matrix $\begin{bmatrix} B\\ \ \one_N\end{bmatrix}$ (in other words, $\hat B_{pr}$ is the initial exchange matrix for the cluster algebra with principal coefficients defined by $B$); and the distance $\mbox{dist}(\hat B_{pr}, \mu^{(n+1)}({\hat B_{pr}})) \leq n$ is defined as the {\em minimal} number of matrix mutations needed 
%to obtain  $\mu^{(n)}({\hat B_{pr}})$ from $\hat B_{pr}$. 
%Note that the condition \eqref{monot_seq} does not depend on $\mathbf{a}$.
%It is {\em determined} by $B$ but needs to be stated in terms of $\hat B_{pr}$ otherwise issues will arise with defining monotonicity, e.g., for mutation finite cluster algebras.

We say that $\mu$ is {\em balanced} if for every $k\in [1,N]$, 
\[
\liminf_{n\to\infty}\frac  {1} {n} \# \{ j\in [1,n] : k_j=k \} > 0.
\]

For an integer $a$, we define
\[
\sgn(a)=\begin{cases} + & a>0\\ 0 & a=0\\ - & a<0 \end{cases}.
\]
We say that a sign vector $\omega=(\omega_1,\dots,\omega_N)\in \{+, 0, -\}^{N}$ is {\em strict} if $\omega_i \neq 0$ for any $i$.
Define {\em the sign vector corresponding to} $\mu^{(n)}$ as
\begin{equation}
\label{sign_vect}
\sigma_\mu^{(n)}(\mathbf a) = \left (\sgn(a_1^{(n)}),\ldots, \sgn(a_N^{(n)})\right ).
\end{equation}
The {\em sign pattern  corresponding to} $\mu$ is defined by
\begin{equation}
\label{sign_seq}
\sigma_\mu(\mathbf a)=\left (\sigma_\mu^{(n)}(\mathbf a)\right )_{n=0}^\infty.
\end{equation}

\begin{conj}[{\bf Asymptotic sign coherence}]
\label{main}
Let $B=\hat B_{[N]}$ be irreducible and let $\mu$ be a monotone and balanced sequence of mutations applied to $\hat B$. Then there exists a sequence of strict sign vectors
\[
\sigma_{\mathrm{reg}}= \left(  \sigma_{\mathrm{reg}}^{(n)}\right )_{n=0}^\infty,  \sigma_{\mathrm{reg}}^{(n)} \in \{\pm\}^N
\]
such that for any nonzero ${\mathbf a} \in \mathbb{Z}^N$ there  is $T\in \mathbb{N}$ such that
%\begin{equation}
%\label{sigma_reg}
$\sigma_\mu^{(n)}(\mathbf a)=\sigma_{\mathrm{reg}}^{(n)}$
%\end{equation}
for $n > T$.
\end{conj}

\begin{rems}
\label{conj_rem}
%1. In \eqref{sigma_reg}, if a certain component of ${\mathbf a}^{(n)}$ is zero, we consider its sign equal to the corresponding component of $\sigma_{\mathrm{reg}}^{(n)}$.

1. If the $C$-matrix part of $\hat B$ has multiple rows (i.e., $M\geq 2$), Conjecture \ref{main} means that for a monotone and balanced sequence of mutations, for all $n$ greater than certain $T$, the sign patterns of any two nonzero rows of the $C$-matrix $C^{(n)}$ obtained on the $n$th step coincide. But that means that the columns of $C^{(n)}$ are sign-coherent. This justifies the term {\em asymptotic sign coherence}.

2. %The distance defined in \eqref{monot_seq} agrees with {\em the distance in the exchange graph of labelled seeds} in the cluster algebra associated with $B$ (with any coefficients). In particular, 
The monotonicity assumption precludes the initial exchange matrix $B$ from being an exchange matrix of a cluster algebra of finite type.

3. Both conditions in Conjecture \ref{main} are used to ensure that the mutation sequence $\mu$ is sufficiently random. Both of these conditions could be relaxed or modified in various ways. In particular,  one could require that a magnitude of every entry of the $C$-matrix tends to infinity as $n$ goes to infinity. One could also weaken the condition of $\mu$ being balanced by replacing it with the requirement that every mutation direction appears in $\mu$ infinitely many times. In this case, we call $\mu$ {\em weakly balanced} (see Remark \ref{rem_Markov} below).
\end{rems}

\section{Rank 2 case}

In this section, we will verify Conjecture \ref{main} in the infinite type  rank $2$ case. 
The matrix $\hat B$ in this case has a form
\begin{equation}
\label{3by2init}
\hat B=\hat B^{(0)}= \begin{bmatrix} 0 & -p \\ q & 0\\ a^{(0)}_1 & a^{(0)}_2\end{bmatrix}, 
%B_{--}  =\begin{bmatrix} 0 & -k \\ k & 0\\ -a_1 & -a_2\end{bmatrix}, B_{+-} =\begin{bmatrix} 0 & -k \\ k & 0\\ a_1 & -a_2\end{bmatrix},
%B_{-+} =\begin{bmatrix} 0 & -k \\ k & 0\\ -a_1 & a_2\end{bmatrix}\ ,
\end{equation}
where positive integer parameters $p,q$ satisfy $pq\geq 4$ \cite[Section 6]{CAI}. We assume that at least one of $a^{(0)}_1, a^{(0)}_2$ is nonzero.

%%%% MENTION PARALLELS WITH LEE-SCHIFFLER / READING%%%%%%%

In the rank $2$ case, the monotonicity assumption for an infinite sequence of mutations amounts to a requirement that mutations $\mu_1$ and $\mu_2$ alternate. Instead of considering two possible monotone sequences of mutations, we will unify them into a single double-infinite sequence. Namely, we consider a sequence of mutations $\mu_1,\mu_2,\mu_1,\mu_2,\ldots$ and denote the exchange matrix obtained at the $n$th step of this sequence by $\hat B^{(n)}$. 
Similarly, the exchange matrix obtained at the $n$th step of a sequence of mutations $\mu_2,\mu_1,\mu_2,\mu_1,\ldots$ will be denoted by $\hat B^{(-n)}$. 
Then, for any $n\in \mathbb{Z}$,
\begin{equation}
\label{3by2n}
\hat B^{(n)}= \begin{bmatrix} 0 & (-)^{n-1}p \\ (-)^{n}q & 0\\ a^{(n)}_1 & a^{(n)}_2\end{bmatrix}.
%B_{--}  =\begin{bmatrix} 0 & -k \\ k & 0\\ -a_1 & -a_2\end{bmatrix}, B_{+-} =\begin{bmatrix} 0 & -k \\ k & 0\\ a_1 & -a_2\end{bmatrix},
%B_{-+} =\begin{bmatrix} 0 & -k \\ k & 0\\ -a_1 & a_2\end{bmatrix}\ ,
\end{equation}
Choosing $\varepsilon=1$ in \eqref{eq:Bmut1}, we obtain the following recursion for $a^{(n)}_1, a^{(n)}_2 (n\geq 0)$:
\begin{equation} 
\label{odd_a+}
\begin{split}
& a^{(2n+1)}_1 = -  a^{(2n)}_1,\\
%\nonumber
& a^{(2n+1)}_2 = a^{(2n)}_2 - \left [-  a^{(2n)}_1 \right ]_+ p,\\
\end{split}
\end{equation}
\begin{equation}
\label{even_a+}
\begin{split}
& a^{(2n+2)}_1 = a^{(2n+1)}_1 - \left [-  a^{(2n+1)}_2 \right ]_+  q,\\
& a^{(2n+2)}_2 = -  a^{(2n+1)}_2.
\end{split}
\end{equation}
For $n\leq 0$, we choose $\varepsilon=-1$, and obtain
\begin{equation}
\label{odd_a-}
\begin{split}
& a^{(-2n-1)}_1 =a^{(-2n)}_1 + \left [a^{(-2n)}_2 \right ]_+  q,\\
& a^{(-2n-1)}_2 = -a^{(-2n)}_2,
\end{split}
\end{equation}
\begin{equation}
\label{even_a-}
\begin{split}
& a^{(-2n-2)}_1 = -  a^{(-2n-1)}_1,\\
& a^{(-2n-2)}_2 = a^{(-2n-1)}_2 + \left [a^{(-2n-1)}_1 \right ]_+  p.
\end{split}
\end{equation}

We want to investigate a dependence of signs of  $a^{(n)}_1, a^{(n)}_2$ on initial values $a^{(0)}_1, a^{(0)}_2$.
As in \eqref{sign_vect}, we define a {\em sign vector} 
$$ 
\sigma^{(n)}=\left(\mbox{sign}(a^{(n)}_1), \mbox{sign}(a^{(n)}_2)\right ).
$$ 
If an expression is only known to be nonnegative (resp. nonpositive), we will denote its sign by $+/0$ (resp. $-/0$).
Following the definition \eqref{sign_seq}, we call a sequence
\begin{equation}
\label{sign_pattern}
\sigma (a^{(0)}_1, a^{(0)}_2) = \left (\sigma^{(n)}\right)_{n\in \mathbb{Z}}
\end{equation}
the {\em sign pattern} with initial values $a^{(0)}_1, a^{(0)}_2$.

We define the sequence $\sigma_{\mathrm{reg}}$ of strict sign vectors as
\begin{equation}
\label{stable}
\sigma_{\mathrm{reg}}  = \left (\left ( (-)^{n-1}, (-)^{n}\right ) \right )_{n\in \mathbb{Z}}.
\end{equation}
We are going to show that for any $a^{(0)}_1, a^{(0)}_2$, the sequence $\sigma (a^{(0)}_1, a^{(0)}_2)$ differs from $\sigma_{\mathrm{reg}}$ for only 
finitely many components.

In order to proceed, we need to recall some of the properties of {\em the Chebyshev polynomials $U_n(t)\ (n\in \mathbb{Z}_{\geq -1})$ of the second kind} (see, e.g. \cite[Chapter 6]{Chebyshev}). They satisfy a three-term recursion of the form
\begin{equation}
\label{cheb_rec}
2 t U_n(t) = U_{n-1}(t) + U_{n+1}(t) \ (n \geq 0);\ U_{-1}(t)=0, U_{0}(t)=1,
\end{equation}
and are orthogonal with respect to the positive weight $\sqrt{1-t^2}$ on the interval $[-1,1]$. As a result, 
\begin{equation}
\label{cheb_pos}
U_n(t) > 0 \ (n\geq 0, t > 1).
\end{equation}
Furthermore, 
\begin{equation}
\label{wronskian}
U^2_n(t) - U_{n-1}(t) U_{n+1}(t) =1\ (n=0,1,\ldots),
\end{equation}
and therefore
\begin{equation}
\label{monotone}
\frac{U_{n}(t)} {U_{n-1}(t)} > \frac{U_{n+1}(t)} {U_{n}(t)} \ (n>0).
\end{equation}
There is an explicit formula for $U_n(t)$,
\begin{equation}
\label{explicit}
U_n(t) = \frac{\left ( t + \sqrt{t^2-1} \right )^{n+1} - \left ( t - \sqrt{t^2-1} \right )^{n+1}} { 2 \sqrt{t^2-1}},
\end{equation}
which implies, in particular, that
\begin{equation}
\label{limit}
\lim_{n\to\infty}\frac{U_{n+1}(t)}{U_{n}(t)} =t + \sqrt{t^2-1} \ (\vert t \vert \geq 1).
\end{equation}

The role played by Chebyshev polynomials in the study of rank 2 cluster algebras was previously observed by several authors, notably in 
\cite{LS1,LS2} where it was utilized in initial steps of the proof of the Laurent positivity conjecture, and, more recently, in \cite{R} where closely related polynomials were used in a description of rank 2 infinite type cluster scattering diagrams. Below, we use Chebyshev polynomials to investigate the sign pattern $\sigma (a^{(0)}_1, a^{(0)}_2)$.

Denote 
$$\kappa = \sqrt{pq}\ ,\  \nu = \sqrt{\frac p q}.
$$

\begin{prop} If $a^{(0)}_1=a_1, a^{(0)}_2=a_2 \geq 0$, then
\label{++}
 \[
 a_1^{(1)} = -a_1, \ a_2^{(1)} = a_2,  a_1^{(2)} = -a_1,\ a_2^{(2)} = -a_2,
 \]
for $n\geq 2$
%{\bf check indices}
\begin{equation}
\label{even_an+}
\begin{split}
& a^{(2n)}_1 = - a_1 U_{2n-2}\left (\frac \kappa 2 \right ) - a_2 \nu^{-1}U_{2n-3}\left (\frac \kappa 2 \right ) < 0,\\
& a^{(2n)}_2 = a_1 \nu U_{2n-3}\left (\frac \kappa 2 \right ) + a_2 U_{2n-4}\left (\frac \kappa 2 \right ) > 0, \\
\end{split}
\end{equation}
for $n\geq 1$
\begin{equation}
\label{odd_an+}
\begin{split}
& a^{(2n+1)}_1 = a_1 U_{2n-2}\left (\frac \kappa 2 \right ) + a_2 \nu^{-1}U_{2n-3}\left (\frac \kappa 2 \right ) \geq 0,\\
& a^{(2n+1)}_2 = - a_1 \nu U_{2n-1}\left (\frac \kappa 2 \right ) - a_2 U_{2n-2}\left (\frac \kappa 2 \right ) < 0, 
\end{split}
\end{equation}
with the first inequality in \eqref{odd_an+} strict for $n>1$,\\
and, for $n\geq 0$
\begin{equation}
\label{even_an-}
\begin{split}
& a^{(-2n-2)}_1 = - a_1 U_{2n}\left (\frac \kappa 2 \right ) - a_2 \nu^{-1}U_{2n+1}\left (\frac \kappa 2 \right ) < 0,\\
& a^{(-2n-2)}_2 = a_1 \nu U_{2n+1}\left (\frac \kappa 2 \right ) + a_2 U_{2n+2}\left (\frac \kappa 2 \right ) > 0,
\end{split}
\end{equation}
\begin{equation}
\label{odd_an-}
\begin{split}
& a^{(-2n-1)}_1 = a_1 U_{2n}\left (\frac \kappa 2 \right ) + a_2 \nu^{-1}U_{2n+1}\left (\frac \kappa 2 \right ) > 0,\\
& a^{(-2n-1)}_2 = - a_1 \nu U_{2n-1}\left (\frac \kappa 2 \right ) - a_2 U_{2n}\left (\frac \kappa 2 \right ) \leq 0, \\
\end{split}
\end{equation}
with the second inequality in \eqref{odd_an-} strict for $n>0$.\\
Therefore, 
$
\sigma^{(-1)} = (+, -/0), \sigma^{(0)} = (+/0, +/0), \sigma^{(1)} = (-/0, +/0),  \sigma^{(2)} = (-/0, -/0)$, $\sigma^{(3)} = (+/0, -)$ and  
\begin{equation}
\label{sign++}
\sigma^{(n)} = \left ( (-)^{n-1}, (-)^{n}\right )  (n\ne -1,0,1,2,3).
\end{equation}
\end{prop}
\begin{proof} By ${\frac \kappa 2} \geq 1$ and \eqref{cheb_pos}, the expressions in \eqref{even_an+}--\eqref{odd_an-} imply the inequalities therein, therefore we obtain  \eqref{sign++}. 

 We will only verify equations in \eqref{odd_an+}. The rest of the formulas in Proposition \ref{++} can be treated the same way. Recursions \eqref{odd_a+}, \eqref{even_a+}
combine into
\begin{equation}
\label{aux3.1}
\begin{split}
& a^{(2n+3)}_1 = -  a^{(2n+1)}_1\ + \left [-  a^{(2n+1)}_2 \right ]_+ q,\\
& a^{(2n+3)}_2 =  -a^{(2n+1)}_2 - \left [a^{(2n+3)}_1 \right ]_+ p .
\end{split}
\end{equation}
For $n=0$, this gives
\begin{align*}
a_1^{(3)} &= a_1 = a_1 U_0\left (\frac \kappa 2\right ) + a_2 \nu^{-1}U_{-1}\left (\frac \kappa 2\right ), \\
a_2^{(3)} &= - a_1 p - a_2 = -a_1 \nu \kappa - a_2= -a_1 \nu U_1\left (\frac \kappa 2\right ) - a_2 U_{0}\left (\frac \kappa 2\right ),
\end{align*}
which is consistent  with \eqref{odd_an+}. We can now proceed by induction while taking into an account that the induction hypothesis stating that \eqref{odd_an+} is valid for $n\leq m$ also ensures that $a_1^{(2n+1)}\geq 0 \geq a_2^{(2n+1)}$ for $n\leq m$.  (Here we use the inequality $\frac \kappa 2 \geq 1$ and \eqref{cheb_pos}.) Then  equations in \eqref{aux3.1} become
\begin{align*}
a^{(2n+3)}_1 &= -  a_1 U_{2n-2}\left (\frac \kappa 2 \right ) - a_2 \nu^{-1}U_{2n-3}\left (\frac \kappa 2 \right )
+ \left (a_1 \nu U_{2n-1}\left (\frac \kappa 2 \right ) + a_2 U_{2n-2}\left (\frac \kappa 2 \right )\right ) q\\
&= a_1 \left (\kappa U_{2n-1}\left (\frac \kappa 2 \right ) -  U_{2n-2}\left (\frac \kappa 2 \right ) \right ) + a_2 \nu^{-1}\left ( - U_{2n-3}\left (\frac \kappa 2 \right ) + \kappa U_{2n-2}\left (\frac \kappa 2 \right )\right )  \\
&\stackrel{\eqref{cheb_rec}}{=} a_1 U_{2n}\left (\frac \kappa 2 \right ) +  a_2 \nu^{-1} U_{2n-1}\left (\frac \kappa 2 \right ), 
\end{align*}
\begin{align*}
a^{(2n+3)}_2 &=   a_1 \nu U_{2n-1}\left (\frac \kappa 2 \right ) + a_2 U_{2n-2}\left (\frac \kappa 2 \right ) -
 \left (a_1 U_{2n}\left (\frac \kappa 2 \right ) + a_2 \nu^{-1}U_{2n-1}\left (\frac \kappa 2 \right )\right ) p\\
&= a_1 \left ( U_{2n-1}\left (\frac \kappa 2 \right ) -  \kappa U_{2n}\left (\frac \kappa 2 \right ) \right )\nu + a_2 \left ( U_{2n-2}\left (\frac \kappa 2 \right ) - \kappa U_{2n-1}\left (\frac \kappa 2 \right )\right )  \\
&\stackrel{\eqref{cheb_rec}}{=} -a_1 \nu U_{2n+1}\left (\frac \kappa 2 \right ) -  a_2 U_{2n}\left (\frac \kappa 2 \right ), 
\end{align*}
which proves \eqref{odd_an+}.
\end{proof}

%\begin{cor}
%\label{--}
%If $a^{(0)}_1=-a_1, a^{(0)}_2=-a_2 \leq 0$ then
%$$
%\sgn(a^{(n-2)}_1)=(-)^{n-1} (n\ne 0,1), \sgn(a^{(-2)}_1)=+, \sgn(a^{(-1)}_1)=- ;
%$$
%$$
%\sgn(a^{(n-2)}_2)=(-)^{n} (n\ne 2), \sgn(a^{(0)}_2)=- \ .
%$$ 
%\end{cor}

Proposition \ref{++} shows that if $a^{(0)}_1, a^{(0)}_2$ are both nonnegative  then the sign pattern $\sigma (a^{(0)}_1, a^{(0)}_2)$ does not depend
on precise values of $a^{(0)}_1, a^{(0)}_2$ and we are justified in using a notation $\sigma _{++}$ for such sign pattern.

Furthermore, if $a^{(0)}_1, a^{(0)}_2$ are both nonpositive, then Proposition \ref{++}  also implies that $\sigma (a^{(0)}_1, a^{(0)}_2) = \left (\sigma_{++}^{(n+2)}\right)_{n\in \mathbb{Z}}$ also does not depend on precise values of $a^{(0)}_1, a^{(0)}_2$ and can be denoted by $\sigma _{--}$.

The case $a^{(0)}_1=a_1 > 0, a^{(0)}_2=-a_2 < 0$ is still covered by Proposition \ref{++}: if one switches the roles of $a^{(n)}_1, a^{(n)}_2$ and also the roles of of $p$ and $q$, then
the corresponding matrix
\[
B^{(0)}= \begin{bmatrix} 0 & q \\ -p & 0\\ -a_2 & a_1\end{bmatrix}
\]
can be seen as ${\tilde B}^{(1)}$ for 
\[
{\tilde B^{(0)}}= \begin{bmatrix} 0 & -q \\ p & 0\\ a_2 & a_1\end{bmatrix}\ 
\]
and so 
\begin{equation}
\label{sigma_+-}
\sigma_{+-}:=\sigma (a^{(0)}_1, a^{(0)}_2) = \left (\tau(\sigma_{++}^{(n+1)})\right)_{n\in \mathbb{Z}},
\end{equation}
where $\tau$ permutes the entries of a two-vector.

The situation is different and more complicated if $a^{(0)}_1=-a_1 < 0 , a^{(0)}_2= a_2 > 0$. In view of \eqref{monotone} and \eqref{explicit}, we separate it into three cases: 
\medskip

Case 1.  $\frac{ 2\nu} { \kappa + \sqrt{\kappa^2 - 4} }\leq  \frac{a_2}{a_1} \leq \frac \nu 2 \left ( \kappa + \sqrt{\kappa^2 - 4}\right )$.
\medskip

Case 2. $\frac \nu 2 \left ( \kappa + \sqrt{\kappa^2 - 4}\right ) < \frac{a_2}{a_1}$.
\medskip

Case 3. $\frac{a_2}{a_1} < \frac{ 2\nu} { \kappa + \sqrt{\kappa^2 - 4} }$.
\medskip

Let us start with Case 1.

\begin{prop}
\label{-+}
 Let  $a^{(0)}_1=-a_1 < 0 , a^{(0)}_2= a_2 > 0$. 

If $\frac{a_2}{a_1} \in \left [ \frac{ 2\nu} { \kappa + \sqrt{\kappa^2 - 4} }, \frac \nu 2 \left ( \kappa + \sqrt{\kappa^2 - 4}\right ) \right ]$, then
for $n\geq 1$, 
\begin{equation}
\label{even_an-++}
\begin{split}
& a^{(2n)}_1 = - a_1 U_{2n}\left (\frac \kappa 2 \right ) + a_2 \nu^{-1}U_{2n-1}\left (\frac \kappa 2 \right ) < 0,\\
& a^{(2n)}_2 = a_1 \nu U_{2n-1}\left (\frac \kappa 2 \right ) - a_2 U_{2n-2}\left (\frac \kappa 2 \right ) > 0,\\
\end{split}
\end{equation}
and for $n\geq 0$,
\begin{equation}
\label{odd_an-++}
\begin{split}
& a^{(2n+1)}_1 = a_1 U_{2n}\left (\frac \kappa 2 \right ) - a_2 \nu^{-1}U_{2n-1}\left (\frac \kappa 2 \right ) > 0,\\
& a^{(2n+1)}_2 = - a_1 \nu U_{2n+1}\left (\frac \kappa 2 \right ) + a_2 U_{2n}\left (\frac \kappa 2 \right )\ < 0,
\end{split}
\end{equation}
\begin{equation}
\label{even_an-+-}
\begin{split}
& a^{(-2n-2)}_1 = a_1 U_{2n}\left (\frac \kappa 2 \right ) - a_2 \nu^{-1}U_{2n+1}\left (\frac \kappa 2 \right ) < 0,\\
& a^{(-2n-2)}_2 = -a_1 \nu U_{2n+1}\left (\frac \kappa 2 \right ) + a_2 U_{2n+2}\left (\frac \kappa 2 \right ) \ > 0,\\
\end{split}
\end{equation}
\begin{equation}
\label{odd_an-+-}
\begin{split}
& a^{(-2n-1)}_1 = -a_1 U_{2n}\left (\frac \kappa 2 \right ) + a_2 \nu^{-1}U_{2n+1}\left (\frac \kappa 2 \right ) > 0,\\
& a^{(-2n-1)}_2 = a_1 \nu U_{2n-1}\left (\frac \kappa 2 \right ) - a_2 U_{2n}\left (\frac \kappa 2 \right ) \ < 0.  \\
\end{split}
\end{equation}
Therefore,
\begin{equation}
\label{sign_reg}
\sigma (a^{(0)}_1, a^{(0)}_2) =\sigma_{\mathrm{reg}},
%\sigma(n) = \left ( (-)^{n-1}, (-)^{n}\right )  \ (n\in \mathbb{Z})\ .
\end{equation}
where $\sigma_{\mathrm{reg}}$ is the one in \eqref{stable}.
\end{prop}
\begin{proof} Our assumption for $\frac{a_2}{a_1}$ together with properties \eqref{monotone} and \eqref{limit} of the Chebyshev polynomials ensure
that
\begin{equation}
\label{btw}
\nu \frac{U_{n}(\frac k 2 )} {U_{n+1}(\frac k 2)} < \frac{a_2}{a_1} < \nu \frac{U_{n+1}(\frac k 2)} {U_{n}(\frac k 2)}
\end{equation}
for any $n > 0$. Then the expressions \eqref{even_an-++} -- \eqref{odd_an-+-} imply the inequalities therein, therefore we obtain \eqref{sign_reg}. As in the proof of Proposition \ref{++}, we verify only one of the expressions \eqref{even_an-++} -- \eqref{odd_an-+-}, for example, \eqref{odd_an-++}. The rest can be treated similarly. First, note that $a_1^{(1)} = a_1=a_1  U_0(\frac \kappa 2) - a_2\nu^{-1} U_{-1}(\frac \kappa 2), a_2^{(1)} = - a_1 p + a_2 =
-a_1 \nu U_1(\frac \kappa 2) + a_2 U_0(\frac \kappa 2)$ satisfy \eqref{odd_an-++} for $n=0$ and that $a_1^{(1)} > 0 >  a_2^{(1)}$. Arguing by induction and using \eqref{aux3.1}, we obtain
\begin{align*}
a^{(2n+3)}_1 &= -  a_1 U_{2n}\left (\frac \kappa 2 \right ) + a_2 \nu^{-1}U_{2n-1}\left (\frac \kappa 2 \right )
+ \left (a_1 \nu U_{2n+1}\left (\frac \kappa 2 \right ) - a_2 U_{2n}\left (\frac \kappa 2 \right )\right ) q\\
%= a_1 \left (\kappa U_{2n-1}\left (\frac \kappa 2 \right ) -  U_{2n-2}\left (\frac \kappa 2 \right ) \right ) + a_2 \nu^{-1}\left ( - U_{2n-3}\left (\frac \kappa 2 \right ) + \kappa U_{2n-2}\left (\frac \kappa 2 \right )\right )  
\\
&\stackrel{\eqref{cheb_rec}}{=} a_1 U_{2n+2}\left (\frac \kappa 2 \right ) -  a_2 \nu^{-1} U_{2n+1}\left (\frac \kappa 2 \right ) > 0,
\end{align*}
\begin{align*}
a^{(2n+3)}_2 &= a_1 \nu U_{2n+1}\left (\frac \kappa 2 \right ) - a_2 U_{2n}\left (\frac \kappa 2 \right )
- \left (a_1 U_{2n+2}\left (\frac \kappa 2 \right ) -  a_2 \nu^{-1} U_{2n+1}\left (\frac \kappa 2 \right )\right ) p\\
%= a_1 \left (\kappa U_{2n-1}\left (\frac \kappa 2 \right ) -  U_{2n-2}\left (\frac \kappa 2 \right ) \right ) + a_2 \nu^{-1}\left ( - U_{2n-3}\left (\frac \kappa 2 \right ) + \kappa U_{2n-2}\left (\frac \kappa 2 \right )\right )  
\\
&\stackrel{\eqref{cheb_rec}}{=} -a_1 \nu U_{2n+3}\left (\frac \kappa 2 \right ) +  a_2 U_{2n+2}\left (\frac \kappa 2 \right ) < 0,
\end{align*}
and the claim follows. Note that here we used the right inequality in \eqref{btw}. It plays the same role in the proof of \eqref{even_an-++}, while for \eqref{even_an-+-}, \eqref{odd_an-+-} the left inequality in \eqref{btw} is needed.
\end{proof}

Next we consider Case 2.
\begin{prop}
\label{-++}
If there exists $N > 0$ such that 
\begin{equation}
\label{ineq-++} 
\nu \frac {U_{N+1}\left (\frac \kappa 2 \right )} {U_{N}\left (\frac \kappa 2 \right )}  \leq \frac {a_2} {a_1} < 
\nu \frac {U_{N}\left (\frac \kappa 2 \right )} {U_{N-1}\left (\frac \kappa 2 \right )},
\end{equation}
then equations \eqref{even_an-+-}, \eqref{odd_an-+-} remain valid for $n \geq 0$, and equations \eqref{even_an-++}, \eqref{odd_an-++} remain valid for
$a^{(k)}_1, a^{(k)}_2 (k=0, 1\ldots, N)$. Furthermore, $a^{(N+1)}_1$ and  $a^{(N+1)}_2$ are nonnegative and not both zero and
\begin{equation}
\label{sign-++}
\sigma (a^{(0)}_1, a^{(0)}_2) = 
\begin{cases} 
\left (\sigma_{++}^{(n-N-1)}\right)_{n\in \mathbb{Z}} & \mbox{for odd}\ N\\
\left (\tau(\sigma_{++}^{(n-N-1)})\right)_{n\in \mathbb{Z}} & \mbox{for even}\ N,
\end{cases}
\end{equation}
where $\tau$ is defined in \eqref{sigma_+-}.
%
%The sign pattern \eqref{sign_reg} remains valid for $n \in \left ( -\infty, N \right ]$.  Furthermore, $a^{(N+1)}_1$ and  $a^{(N+1)}_2$ are nonnegative and not both zero. Thus, the %evolution and sign pattern of $a^{(n)}_1$ and  $a^{(n)}_2$ for $n\geq N+1$ is governed by Proposition \ref{++}.
\end{prop}
\begin{proof} In the proof of Proposition \ref{-+}, our argument relied on the inequalities \eqref{btw}. Under the current assumptions, the left inequality remains valid for all $n$, while the right inequality is valid for $n<N$. This explains the claim about the validity of \eqref{even_an-++} -- \eqref{odd_an-+-} in this situation. 

Next, if $N$ is odd, $N=2 m + 1$, then \eqref{even_an-++}, \eqref{odd_an-++} are valid for $n=m$ with $a_1^{(2m)} < 0 < a_2^{(2m)}$ and $a_1^{(2m+1)} >  0 > a_2^{(2m+1)}$.
This means that the expressions in \eqref{even_an-++} are also valid for $n=m+1$, however, due to \eqref{ineq-++},  $a_1^{(2m+2)}\geq 0$ and  $a_2^{(2m+2)} > 0$. This puts us in the situation covered by Proposition \ref{++} and the first line in \eqref{sign-++} follows. The case of $N$ even is treated in the same way.
\end{proof}

Finally, we consider Case 3.
\begin{prop}
\label{--+}
If there exists $N > 0$ such that 
$$ \nu \frac {U_{N-1}\left (\frac \kappa 2 \right )} {U_{N}\left (\frac \kappa 2 \right )}  \leq \frac {a_2} {a_1} < 
\nu \frac {U_{N}\left (\frac \kappa 2 \right )} {U_{N+1}\left (\frac \kappa 2 \right )},
$$ 
then equations \eqref{even_an-++}, \eqref{odd_an-++} remain valid for $n \geq 0$, and equations \eqref{even_an-+-}, \eqref{odd_an-+-} remain valid for
$a^{(k)}_1, a^{(k)}_2 (k=0, -1\ldots, -N)$. Furthermore, $a^{(-N-1)}_1$ and  $a^{(-N-1)}_2$ are nonpositive and not both zero, and 
\[
\sigma (a^{(0)}_1, a^{(0)}_2) = 
\begin{cases} 
\left (\sigma_{++}^{(n+N+3)}\right)_{n\in \mathbb{Z}} & \mbox{for odd}\ N\\
\left (\tau(\sigma_{++}^{(n+N+3)})\right)_{n\in \mathbb{Z}} & \mbox{for even}\ N,
\end{cases}
\]
where $\tau$ is defined in \eqref{sigma_+-}.

%The sign pattern \eqref{sign_reg} remains valid for $n \in \left [ -N, \infty \right )$.  Furthermore,  $a^{(-N-1)}_1$ and  $a^{(-N-1)}_2$ are nonpositive and not both zero. 
%Thus, the evolution and sign pattern of $a^{(n)}_1$ and  $a^{(n)}_2$ for $n\leq -N-1$ is governed by Proposition \ref{++}.
\end{prop}
\begin{proof} The proof is completely analogous to that of a previous proposition.
\end{proof}

Combining Propositions \ref{++}--\ref{--+}, we arrive at the following conclusion.

\begin{thm}
\label{rank2}
Conjecture \ref{main} is valid in the rank $2$ case.
\end{thm}

\begin{proof} One can see that each of the sign patterns $\sigma_{++}, \sigma_{--}, \sigma_{+-}$,  as well as every sign pattern that appears in Propositions \ref{-++}, \ref{--+}, differs from  $\sigma_{\mathrm{reg}}$ in \eqref{stable} for exactly three consecutive components.
%It is sufficient to show that for any $a^{(0)}_1, a^{(0)}_2$, $\sigma (a^{(0)}_1, a^{(0)}_2)$ differs from $\sigma_{\mathrm{reg}}$ only for finitely many $n$. By Proposition \ref{++}, this is true for $\sigma_{++}$. Since $\sigma_{--}$ as well as  $\sigma (a^{(0)}_1, a^{(0)}_2)$  described in Propositions \ref{-++},  \ref{--+} in the case of odd $N$, are obtained from  $\sigma_{++}$ by an even shift, the claim is true for these sign patterns, too. Finally, applying the permutation $\tau$ to elements of  $\sigma_{\mathrm{reg}}$ results in a shift by one and so, by \eqref{sigma_+-}, $\sigma_{+-}$differs from $\sigma_{st}$ only in finitely many elements. This, together with the even $N$ case of Propositions \ref{-++},  \ref{--+}, completes the proof.
\end{proof}

\section{Rank $3$ example: the Markov quiver}

To present additional evidence in support of Conjecture \ref{main}, we consider the rank $3$ case with $B$ being the adjacency  matrix of the celebrated {\em Markov quiver}.  The corresponding cluster algebra served as a test case for several important phenomena in the theory of cluster algebras. In particular, the principal $c$-vectors and $g$-vectors associated with the Markov quiver were described in \cite{NC}.

We add  an additional frozen vertex to the Markov quiver and investigate possible sign patterns, for example, for a monotone and balanced sequence 
\[
\mu = \left (1,2,3,1,2,3,\ldots \right ).
\] 
The figure below illustrates the case when the initial vector ${\mathbf a}^{(0)}=(a_1,a_2,a_3)$ is componentwise nonnegative. In this figure, we assume that there are $a_i$ arrows pointing from the frozen vertex to the vertex $i$.
%\begin{figure}[h]
%\begin{center}
%\includegraphics[width=5cm, height=5cm]{Markov.pdf}
%\end{center}
%\caption{}
%\label{Markov_quiver}
%\end{figure}

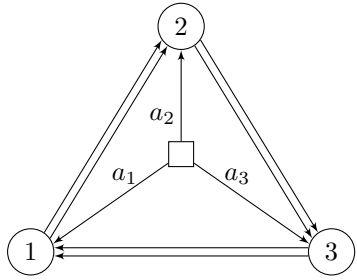
\begin{center}
\begin{tikzpicture}
\tikzset{vertex/.style = {shape=circle,draw,minimum size=1.5em}}
\tikzset{edge/.style = {->,> = latex'}}
% vertices
\node[vertex] (1) at  (0,0) {$1$};
\node[vertex] (2) at  (2,3) {$2$};
\node[vertex] (3) at  (4,0) {$3$};
\node[regular polygon,regular polygon sides=4,draw]  (4) at  (2,1.3) {};

\node[rectangle] (5) at (1.25,1) {$a_1$};
\node[rectangle] (6) at (1.75,1.8) {$a_2$};
\node[rectangle] (7) at (2.75,1) {$a_3$};

%edges
\draw[edge] (1.35) to (2.235);
\draw[edge] (1.55) to (2.215);
\draw[edge] (3.170) to (1.10) ;
\draw[edge] (3.190) to (1.350) ;
%\draw[edge] (1.325) to (4.125);
\draw[edge] (2.305) to (3.145);
\draw[edge] (2.325) to (3.125);
\draw[edge] (4) to (2);

\draw[edge] (4) to (3.160);
\draw[edge] (4) to (1.20);
\end{tikzpicture}
\end{center}

Since the sequence $\mu$ consists in a consecutive mutations applied to the quiver in a clockwise order starting with the vertex 1 and since, as is well-known, the Markov quiver transforms into its opposite after a mutation at any vertex, we can replace $\mu$ with iterations of {\em the same} transformation $\rho$ that consists of a quiver mutation at the vertex 1, followed by rotation of the quiver counter clockwise by 120 degrees (equivalently, applying a cyclic permutation $\tau=(132)$ to the mutable vertices) and then reversing all arrows. Using \eqref{eq:Bmut1}, we compute the result of $\rho$ acting on the initial nonzero vector ${\mathbf a}^{(0)}=(a^{(0)}_1,a^{(0)}_2,a^{(0)}_3)\in \mathbb{Z}^3$:
\begin{equation}
\label{rho}
\rho ({\mathbf a}^{(0)}) =\left (\rho({\mathbf a}^{(0)})_1, \rho({\mathbf a}^{(0)})_2, \rho({\mathbf a}^{(0)})_3\right )= \left (-2[a^{(0)}_1]_+ - a^{(0)}_2, 2[-a^{(0)}_1]_+ - a^{(0)}_3, a^{(0)}_1\right ).
\end{equation}

For a (not necessarily strict) sign vector $\tilde\omega=(\tilde\omega_1,\tilde\omega_2,\tilde\omega_3)$ and a strict sign vector $\omega=(\omega_1,\omega_1,\omega_3)$,
we write $\tilde\omega\approx\omega$ if 
%\begin{equation*}
%\label{obey}
$(\tilde\omega_i,\omega_i)\neq (+,-), (-,+)\  \text{for any}\  i$.
%\end{equation*}

Observe that if ${\omega}^{(0)} =(- , +, -)$ then $\omega^{(n)}=\omega^{(0)}$ and so $\omega^{(n)}$ is stable. More generally, if ${\omega}^{(0)}\approx (- , +, -)$, i.e., ${\mathbf a}^{(0)}=\left (-a_1, a_2, -a_3\right)$, where $a_1, a_2, a_3$ are nonnegative integers and not all zero, then $\omega^{(n)}\approx (- , +, -)$ as well and, moreover,  $\omega^{(n)}=(- , +, -)$ for $n\geq 5$. This follows from a straightforward computation that gives
\[
\rho^5 ({\mathbf a}^{(0)}) = \left ( -4a_1 - 4a_2 - a_3, 9a_1 + 4a_2 + 4a_3, -4a_1 - a_2 - 2a_3 \right ).
\]
Thus, to establish Conjecture \ref{main} for $\mu = \left (1,2,3,1,2,3,\ldots \right )$, it suffices to show that $\omega^{(n)}=\omega^{(n)}({\mathbf a}^{(0)}) $  eventually stabilizes  at $(-, +, -)$ for any ${\mathbf a}^{(0)}$.

Denote by $\omega^{(n)}=\omega^{(n)}({\mathbf a}^{(0)}) $ the sign vector of $\rho^n({\mathbf a}^{(0)}) $ defined similarly to \eqref{sign_vect}.
For the rest of this section, we fix $a_1, a_2, a_3$ to be nonnegative integers and not all zero. 
%All considerations below can be easily adapted to the case when these numbers are nonnegative but not all zero. 
In what follows, the choice of ${\omega}^{(0)}\approx (\epsilon_1 , \epsilon_2, \epsilon_3)$ signifies that ${\mathbf a}^{(0)} =\left (\epsilon_1 a_1, \epsilon_2 a_2, \epsilon_3 a_3\right )$.

Let us first consider the case when  ${\omega}^{(0)}\approx (+ , +, -)$, which leads to 
\begin{align*}
&\rho\left({\mathbf a}^{(0)}\right )= \left (-2a_1 -a_2, a_3, a_1 \right ), &{\omega}^{(1)}\approx (-, +, +),\\
&\rho^2\left({\mathbf a}^{(0)}\right )= \left (- a_3, 3a_1 + 2a_2, -2a_1 - a_2 \right ), &{\omega}^{(2)}\approx(-, +, -),
\end{align*} 
after which the desired sign stabilization occurs. 
Similarly, starting with ${\omega}^{(0)}\approx (- , -, -)$, we obtain
\begin{align*}
&\rho\left({\mathbf a}^{(0)}\right )= \left (a_2, 2a_1 + a_3, - a_1 \right ), &{\omega}^{(1)}\approx(+, +, -),
\end{align*}
which reduces to the case we just considered, as does the
situation depicted in the figure above that corresponds to ${\omega}^{(0)}\approx (+ , +, +)$, which results in
\begin{align*}
&\rho\left({\mathbf a}^{(0)}\right )= \left (-2a_1 - a_2, -a_3, a_1 \right ), &{\omega}^{(1)}\approx (-, -, +),\\
&\rho^2\left({\mathbf a}^{(0)}\right )= \left (a_3, 3a_1 + 2a_2, -2a_1 - a_2 \right ), &{\omega}^{(2)}\approx (+, +, -).
%&\rho^3\left({\mathbf a}^{(0)}\right )= \left (-3a_1- 2a_2-2a_3, 2a_1 + a_2, a_3 \right ), &{\omega}^{(3)} =(-, +, +)\ ,\\
%&\rho^4\left({\mathbf a}^{(0)}\right )=\left (-2a_1 - a_2, 6a_1+4a_2+3a_3, -3a_1- 2a_2-2a_3 \right ), &{\omega}^{(4)} =(-, +, -)\ ,
\end{align*} 

For ${\omega}^{(0)}\approx (- , +, +)$, we get $\rho\left({\mathbf a}^{(0)}\right )= \left (- a_2, 2 a_1-a_3, -a_1 \right )$ and so ${\omega}^{(1)}\approx (- , +, -)$ or $(- , -, -)$, depending on the sign of $2 a_1-a_3$. Both of these cases were already covered above.

We are left with three remaining choices for ${\omega}^{(0)}$ : $(+ , -, -/0), (-/0 , -, +)$ and $(+ , -, +)$. The first of these, depending on the sign of $-2a_1 + a_2$, results in ${\omega}^{(1)}=(+ , +, +)$ or in ${\omega}^{(1)}=(-/0, +, +)$ --- both situations have been already treated above. 

The case ${\omega}^{(0)}= (-/0 , -, +)$, depending on the sign of $2a_1 -a_3$, leads to ${\omega}^{(1)}=(+ , +, -)$ or in ${\omega}^{(1)}=(+ , -/0, -)$, also covered by now.

The last case requires a more delicate analysis. Indeed, if ${\mathbf a}^{(0)}=(a_1,-a_2,a_3)$ and if $a_2 > 2 a_1$, then the sign pattern
for $ \rho\left({\mathbf a}^{(0)}\right )= \left (-2a_1 + a_2, -a_3, a_1 \right )$ is also $(+ , -, +)$. Persistence of such situation would contradict our claim that ${\omega}^{(n)}$ stabilizes at $(-, +, -)$.

\begin{lem}
\label{fib}
Suppose $a_1, a_2, a_3$ are positive numbers such that for ${\mathbf a}^{(0)}=(a_1,-a_2,a_3)$, ${\omega}^{(n)}= (+, -, +)$ for all $n$.
Then the first component  of $ \rho^n\left({\mathbf a}^{(0)}\right )$ is 
\begin{equation}
\label{fib_rec}
\rho^n\left({\mathbf a}^{(0)}\right )_1 = (-1)^n \left (   (f_{n+3} -1) a_1 -  (f_{n+2} -1) a_2 +    (f_{n+1} -1)  a_3     \right ),
\end{equation}
where $f_n$ denotes the $n$th Fibonacci's number.
\end{lem}
\begin{proof} The claim \eqref{fib_rec} is checked directly for $n=0,1,2$, where 
\[
\left({\mathbf a}^{(0)}\right )_1=a_1,  \rho\left({\mathbf a}^{(0)}\right )_1=-2a_1 + a_2, \rho^2\left({\mathbf a}^{(0)}\right )_1= 4a_1 -  2a_2 + a_3
\]
and $f_1=f_2=1, f_3=2, f_4=3, f_5=5$.

Under the assumptions of the lemma, \eqref{rho} implies 
\[
\rho^{n+1}\left({\mathbf a}^{(0)}\right )_2 = - \rho^{n}\left({\mathbf a}^{(0)}\right )_3  = - \rho^{n-1}\left({\mathbf a}^{(0)}\right )_1\ 
\]
and then
\[
\rho^{n+2}\left({\mathbf a}^{(0)}\right )_1 = - 2 \rho^{n+1}\left({\mathbf a}^{(0)}\right )_1 + \rho^{n-1}\left({\mathbf a}^{(0)}\right )_1.
\]
The claim follows by induction from a relation $2 f_{n+2} - f_{n} = f_{n+3}$ which is an easy consequence of the Fibonacci recursion.
\end{proof}
As a corollary of Lemma \ref{fib}, we conclude that in order for a sign pattern $(+, -, +)$ to persists, the constants $a_1, a_2, a_3$ must satisfy inequalities
\begin{equation}
\label{markov_ineq}
(f_{2n+1} -1) a_1 -  (f_{2n} -1) a_2 +    (f_{2n-1} -1)  a_3   > 0 >  (f_{2n+2} -1) a_1 -  (f_{2n+1} -1) a_2 +    (f_{2n} -1)  a_3 
\end{equation}
for all $n>0$. Since $\lim_{n\to \infty}\frac{f_{n+1}}{f_n}= \varphi$, where $\varphi=\frac{1+\sqrt{5}}{2}$ is the {\em golden ratio}, \eqref{markov_ineq} implies that 
$a_1 \varphi^2 - a_2 \varphi + a_3 = \frac 1 2 \left ( (3a_1 - a_2 + 2 a_3) + (a_1 - a_2)\sqrt{5}\right )=0$. This is where the integrality of $a_1, a_2, a_3$ comes into play, since for the last equation to hold, we must have $a_1 = a_2 = -a_3$, which contradicts our positivity assumption for $a_1, a_2, a_3$. The conclusion is that if ${\omega}^{(0)}=(+ , -, +)$ then there exists such $n$ that ${\omega}^{(n)}=(-/0 , -, +)$. This concludes the proof of Conjecture \ref{main} for the Markov quiver and $\mu = \left (1,2,3,1,2,3,\ldots \right )$.

\begin{rem}
\label{rem_Markov}
%The Markov quiver can also be used to illustrate the necessity of 
The condition of a sequence of mutations $\mu$ being at least weakly balanced is necessary for Conjecture \ref{main} to hold true in the case of the Markov quiver. Indeed, let $\mu= (1, 2, 1, 2, ....)$ and ${\mathbf a}^{(0)} = (1, -1, a)$, where $a$ is any integer. Then it is easy to check that ${\mathbf a}^{(n)} = \left ((-1)^n, (-1)^{n+1}, a\right )$, 
${\sigma}^{(n)} = \left ((-1)^n, (-1)^{n+1}, \sgn{(a)}\right )$ and the claim in Conjecture \ref{main} fails.
\end{rem}

%Similarly, starting with ${\mathbf a}^{(0)}=(-a_1,-a_2,-a_3), {\omega}^{(0)} =(- , -, -)$, we obtain
%\begin{align*}
%&\rho\left({\mathbf a}^{(0)}\right )= \left (a_2, 2a_1 + a_3, - a_1 \right ), &{\omega}^{(1)} =(+, +, -)\ ,\\
%&\rho^2\left({\mathbf a}^{(0)}\right )= \left (-2a_1 - 2a_2 - a_3, a_1, a_2 \right ), &{\omega}^{(2)} =(-, +, +)\ ,\\
%&\rho^3\left({\mathbf a}^{(0)}\right )= \left (-a_1, 4 a_1 + 3 a_2 + 2 a_3, -2a_1 - 2a_2 - a_3 \right ), &{\omega}^{(3)} =(-, +, -)\ 
%\end{align*} 
%and starting with ${\mathbf a}^{(0)}=(a_1,a_2,-a_3), {\omega}^{(0)} =(+ , +, -)$, we obtain
%\begin{align*}
%&\rho\left({\mathbf a}^{(0)}\right )= \left (-2a_1 -a_2, a_3, a_1 \right ), &{\omega}^{(1)} =(-, +, +)\ ,\\
%&\rho^2\left({\mathbf a}^{(0)}\right )= \left (- a_3, 3a_1 + 2a_2, -2a_1 - a_2 \right ), &{\omega}^{(2)} =(-, +, -)
%\end{align*} 

%\bibliography{../../biblist/biblist.bib}

\end{document}